\documentclass{amsart}
\usepackage{amsmath,amsthm,amsfonts,amssymb,amstext, latexsym, amscd}

\newtheorem{theorem}{Theorem}[section]
\newtheorem{lem}[theorem]{Lemma}
\newtheorem{prop}[theorem]{Proposition}

\newtheorem{question}[theorem]{Question}
\theoremstyle{definition}

\theoremstyle{remark}

\numberwithin{equation}{section}

\newcommand{\ra}{\rightarrow}

\newcommand{\Soul}{\Sigma}

\newcommand{\lb}{\langle}
\newcommand{\rb}{\rangle}

\newcommand{\R}{\mathbb{R}}

\newcommand{\og}{\overline{\gamma}}
\newcommand{\ox}{\overline{X}}
\newcommand{\ov}{\overline{V}}
\newcommand{\ou}{\overline{U}}

\begin{document}

\title{Rigidity for Odd-dimensional Souls}
\author{Kristopher Tapp}
\address{Department of Mathematics\\ Saint Joseph's University\\ 5600 City Avenue
         Philadelphia, PA 19131}
\email{ktapp@sju.edu}

\subjclass{Primary 53C20}
\date{\today}
\keywords{nonnegative curvature, soul, flats}

\begin{abstract}
We prove a new rigidity result for an open manifold $M$ with nonnegative sectional curvature whose soul $\Sigma\subset M$ is odd-dimensional.  Specifically, there exists a geodesic in $\Sigma$ and a parallel vertical plane field along it with constant vertical curvature and vanishing normal curvature.  Under the added assumption that the Sharafutdinov fibers are rotationally symmetric, this implies that for small $r$, the distance sphere $B_r(\Soul)=\{p\in M\mid\text{dist}(p,\Soul)=r\}$ contains an immersed flat cylinder, and thus could not have positive curvature.
\end{abstract}

\maketitle
\date{\today}

\section{Introduction}
In this paper, we prove the following rigidity result for odd-dimensional souls:
\begin{theorem}\label{A}
 If $M$ is an open manifold with nonnegative sectional curvature whose soul $\Sigma\subset M$ is odd-dimensional, then there exists a geodesic $\gamma(t)$ in $\Sigma$, and orthonormal parallel vertical vector fields $V(t),W(t)$ along $\gamma(t)$ (``vertical'' means orthogonal to $T_{\gamma(t)}\Sigma$ for each $t$) such that $\lb R(V(t),W(t))W(t),V(t)\rb$ is constant and $R(V(t),W(t))\gamma'(t)=0$ for all $t$, where $R$ denotes the curvature tensor of $M$.
\end{theorem}

We will discuss why this rigidity provides infinitesimal evidence of an affirmative answer to:
\begin{question}\label{Q} For each small $r>0$, must the image, $C_r$, of the immersed cylinder $$\sigma(t,\theta)=\{\exp_{\gamma(t)}\left( (r\cos\theta)V(t)+(r\sin\theta)W(t)\right)\mid t\in\R,\theta\in S^1\}$$ be ``flat'' in the sense that the Jacobi field $\frac{\partial\sigma}{\partial\theta}$ along each of the $\{\theta=\text{constant}\}$ geodesics out of which $C_r$ is ruled is a parallel Jacobi field in $M$?
\end{question}

Notice that $C_r$ is not totally geodesic in $M$, but might be totally geodesic in the distance sphere $B_r(\Soul)=\{p\in M\mid\text{dist}(p,\Soul)=r\}$.  Recall that $B_r(\Soul)$ is convex in $M$ by~\cite{GW} and thus always inherits nonnegative curvature.  An affirmative answer to Question~\ref{Q} means that $B_r(\Soul)$ could \emph{not} have strictly positive curvature when $\Soul$ is odd-dimensional.  However, examples are known for which $B_r(\Soul)$ has \emph{some} points of positive curvature.  Specifically,  Wilking constructed in~\cite{W} a metric with almost positive curvature on $S^2\times S^3$ that can be extended to a nonnegatively curved metric on $S^3\times \R^3$ (see~\cite{T3} for an alternative description of his metric that makes the extendability more obvious).

We provide an affirmative answer to Question~\ref{Q} under the added hypothesis that the intrisic metric on each Sharafutdinov fiber is rotationally symmetric (which means $O(k)$-invariant, where $k$ is the dimension of the fiber).  We do \emph{not} need to assume that the fibers all have the same rotationally symmetric metric.

\begin{theorem}\label{B}
With the assumptions and terminology of Theorem~\ref{A} and Question~\ref{Q}, if the Sharafutdinov fibers are all rotationally symmetric, then for each $r>0$, $C_r$ is flat; further, the fibers along $\gamma$ are mutually isometric.
\end{theorem}

We are pleased to thank Ilya Bogdanov for the proof of Lemma~\ref{ilya}.  We would also like to thank Igor Belegradek and Luis Guijarro for helpful discussions about this work.

\section{Background}
For the remainder of this paper, $M$ will denote an open manifold with nonnegative curvature.  According to~\cite{CG}, $M$ is diffeomorphic to the total space of the normal bundle of its soul, $\Sigma\subset M$.  We will denote this normal bundle as $\nu(\Soul)$, and its fiber at $p\in\Soul$ as $\nu_p(\Soul) = \{V\in T_p M\mid V\perp T_p\Soul\}$.  Our main tool is the following ``soul inequality'' for the curvature tensor, $R$, of $M$, found in~\cite{T}:

\begin{prop}[\cite{T}] \label{soulin}For all $p\in\Soul$, $X,Y\in T_p\Soul$ and $V,W\in\nu_p(\Soul)$, we have:
$$(D_XR)(X,Y,W,V)^2 \leq \left( |R(W,V,X)|^2 + \frac 23 (D_XD_XR)(W,V,V,W)\right)\cdot R(X,Y,Y,X).$$
\end{prop}

 Here, we are considering $R$ sometimes as a function from  $(T_pM)^3\ra T_pM$ and sometimes from $(T_pM)^4\ra\R$, in the obvious way.  This inequality was originally expressed in~\cite{T} in a manner which more explicitly distinguished the three different types of curvature that it relates:
$$\lb (D_XR^\nabla)(X,Y)W,V\rb^2 \leq\left(|R^\nabla(W,V)X|^2 + \frac 23 (D_XD_X k^f)(W,V)\right)\cdot k_\Soul(X,Y).$$
Here, $k_\Soul$  and $k^f$ denote respectively the unnormalized intrinsic sectional curvature of $\Soul$ and of the Sharafutinov fiber, $\exp(\nu_p(\Soul))$.  Notice that the intrinsic equals the extrinsic curvature because $\Soul$ is totally geodesic, and because each Sharafutinov fiber is totally geodesic a point of the soul.  To interpret the $k_f$ term, just extend $W,V$ to parallel fields $W(t),V(t)$ along the geodesic in $\Soul$ in the direction of $X$, and notice that $(D_XD_X k^f)(W,V)$ equals the second derivative at $t=0$ of the vertical curvature function $t\mapsto k^f(W(t),V(t))$.

Further, $R^\nabla:T_p\Soul\times T_p\Soul\times\nu_p(\Soul)\ra\nu_p(\Soul)$ denotes the ``normal curvature tensor'' which means the curvature tensor of the induced connection, $\nabla$, in $\nu(\Soul)$, so that $R^\nabla(W,V)X\in T_p\Soul$ can be defined as the unique vector such that $\lb R^\nabla(W,V)X,Y\rb = \lb R^\nabla(X,Y)W,V\rb$ for all $Y\in T_p\Soul$.  The fact that $R^\nabla$ is just a restriction of $R$ follows from the fact that $\Soul$ is totally geodesic.
\section{Proof of Theorem~\ref{A}}
We will require the following fact about smooth functions:
\begin{lem}\label{ilya}
Suppose $f,g:\R\ra\R$ are smooth functions.  Assume that $f(0)=0$ and that $g(t)$ has a global maximum at $t=0$.    Assume for all $t\in\R$ that:
$$f'(t)^2 \leq f(t)^2 + g''(t).$$
Then $f$ and $g$ are both constant functions.
\end{lem}
The following proof is due to Ilya Bogdanov, communicated via \emph{mathoverflow.net}.  Notice that the Lemma is very simple to prove for analytic functions.
\begin{proof}
Assume without loss of generality that $g(0)=0$.  Using Cauchy-Schwarz, we have that for all $t\in(0,1)$:
$$\int_0^t\left(f(s)^2 + g''(s)\right)\,\text{ds} \geq \int_0^t f'(s)^2\,\text{ds}
\geq\frac{\left(\int_0^t f'(s)\,\text{ds}\right)^2}{\int_0^t 1^2\,\text{ds}} =
\frac{f(t)^2}{t} \geq f(t)^2.$$
Therefore,
$$g'(t)\geq f(t)^2 - \int_0^t f(s)^2\,\text{ds}.$$
Define $h(t) = \int_0^t f(s)^2\,\text{ds}$, so the previous equation becomes:
$$g'(t)\geq h'(t)-h(t).$$
Since $h$ is monotonic, integrating the previous equation equation gives:
$$g(t)\geq h(t) - \int_0^t h(s)\,\text{ds}\geq h(t)-th(t)\geq 0.$$
Since $g(0)=0$ is a global maximum, this implies that $g(t)=h(t)=f(t)=0$ for all $t\in(0,1)$, and thus clearly also for all $t\in\R$.
\end{proof}

We are now prepared to prove our first theorem.
\begin{proof}[Proof of Theorem~\ref{A}]
Chose $p\in\Soul$ and orthogonal unit-length vectors $W,V$ in $\nu_p(\Soul)$ such that $k^f(W,V)$ is maximal (among all such $p,V,W$), which implies that $(D_XD_X k^f)(W,V)=0$ for all $X\in T_p \Soul$.  Since $X\mapsto R^\nabla(W,V)X$ is a skew-symmetric endomorphism of the odd-dimensional vector space $T_p \Soul$, there exists a unit-length vector $X\in T_p \Soul$ such that $R(W,V)X=0$.  For any $Y\in T_p \Soul$, the right side of the soul inequality vanishes for the vectors $\{X,Y,W,V\}$, and therefore the left side also vanishes.  

Let $\gamma(t)$ denote the geodesic in $\Soul$ with $\gamma(0)=p$ and $\gamma'(0)=X$.  Let $X(t) = \gamma'(t)$, which is the parallel transport of $X$ along $\gamma(t)$.  Let $W(t),V(t)$ denote the parallel transports of $W,V$ along $\gamma(t)$.  Define:
$$g(t) = \frac 23 \cdot k^f(V(t),W(t))\,\,\,\text{ and }\,\,\, f(t) = |R^\nabla(W(t),V(t))X(t)|.$$
Let $C$ denote the maximum sectional curvature of $\Soul$.  For any unit-length $Y(t)\in T_{\gamma(t)}\Soul$, the soul inequality gives:
$$\Big\langle \frac{D}{dt}(R^\nabla(W(t),V(t))X(t)),Y(t)\Big\rb^2 \leq\left(f(t)^2 + g''(t)\right)\cdot C.$$
In particular, choosing $Y$ parallel to $\frac{D}{dt}(R^\nabla(W(t),V(t))X(t))$ gives:
$$(f'(t))^2 \leq ( f(t)^2 + g''(t) )\cdot C.$$
Lemma~\ref{ilya} now implies that $f$ is identically zero and that $g$ is constant.
\end{proof}

\section{Proof of Theorem~\ref{B}}
In this section, we prove our second theorem.
\begin{proof}[Proof of Theorem~\ref{B}]
Let $\gamma(t),X(t),V(t)$ and $W(t)$ be as in the proof of Theorem~\ref{A}.  Choose a fixed $r>0$.
Let $\exp^\perp:\nu(\Sigma)\ra M$ denote the normal exponential map.  Let $A$ and $T$ denote the fundamental tensors of the Sharafutdinov map $\pi:M\ra\Soul$.

Let $\og(t) = \exp^\perp(r\cdot W(t))$, which is one of the $\pi$-horizontal geodesics out of which the cylinder $C_r$ is ruled.  In fact, since $\{W,V\}$ can be replaced with any orthonormal basis of their span, we can consider $\og(t)$ as an \emph{arbitrary} one of the geodesics which rule $C_r$.

Let $\ox(t) = \og'(t)$ and let $\ov(t) = (\exp^\perp_*)(V(t))$, where $(\exp^\perp_*)$ denotes the derivative of $\exp^\perp$ at the relevant point, which in this case is $r\cdot W(t)$.  Notice that $\ov(t)$ is a $\pi$-vertical vector field along $\og(t)$.  In fact, $\ov(t)$ is a vertical Jacobi field along $\og(t)$ because it is the variational field of the family of horizontal geodesics by which the cylinder $C_r$ is ruled.  The covariant derivative of $\ov(t)$ along $\og(t)$ has horizontal and vertical components determined respectively by the $A$ and $T$-tensors of $\pi$:
\begin{equation}\label{willow}\ov'(t) = A_{\ox(t)}\ov(t) + T_{\ov(t)}\ox(t).\end{equation}
To prove that $C_r$ is flat (in the sense of Question~\ref{Q}) it will suffice to show that $\ov'(t)=0$ for all $t$.  In fact, it will suffice to prove this just for $t=0$.

We will sometimes write $X,V,W,\ox$ and $\ov$ for the values of these fields at $t=0$.  By the Jacobi Equation, the sectional curvature of the vertizontal plane spanned by $\ox$ and $\ov$ equals:
\begin{equation}\label{kv} k(\ox,\ov)  = - \lb \ov'',\ov\rb. \end{equation}

We first claim that $A_{\ox}\ov =0$.  Recall that for any horizontal vectors $\overline{Z}_1,\overline{Z}_2$ at $\og(0)$, Proposition 1.7 of~\cite{SW2} implies that the $A$-tensor can be described as:
$$A_{\overline{Z}_1}\overline{Z}_2 =  \frac 12 (\exp^\perp_*)(R^\nabla(\pi_*\overline{Z}_1,\pi_*\overline{Z}_2)W).$$
For any horizontal vector $\overline{Y}$ at $\overline{\gamma}(0)$, let $Y=\pi_*\overline{Y}$ and notice that:
$$\lb A_{\ox}\ov,\overline{Y}\rb = -\lb A_{\ox}\overline{Y},\ov\rb
   = -\frac 12 \lb (\exp^\perp_*)(R^\nabla(X,Y)W),(\exp^\perp_*)V \rb.$$
This equals zero because
$$\lb R^\nabla(X,Y)W,V\rb = \lb R^\nabla(W,V)X,Y\rb = 0,$$
and because $\exp^\perp_*$ preserves orthogonality due to our assumption that the fibers are rotationally symmetric (and thus that the distance spheres in these fibers are round).

Since the $A$-tensor term  vanishes, Equation~\ref{willow} becomes $\ov'(t) = T_{\ov(t)}\ox(t)$, and it remains to prove that this $T$-tensor term vanishes.  Let $\ou$ be an arbitrary vertical vector at $\og(0)$ which is perpendicular to $\ov$ and to the radial direction.  Let $\ou(t)$ denote the extension of $\ou(0)=\ou$ to the ``hononomy Jacobi field'' along $\og(t)$; that is, the variational vector field of the family of lifts of $\gamma(t)$ to horizontal geodesics beginning at points along a curve in the Sharafutdinov fiber tangent to $\ou$.  Notice that for each $t_0$, $\ov(t_0)$ and $\ou(t_0)$ are the images of $\ov$ and $\ou$ under the derivative at $\og(0)$ of the ``holonomy diffeomorphism" $h_\gamma:\pi^{-1}(\gamma(0))\ra\pi^{-1}(\gamma(t_0))$.  Since the fibers are rotationally symmetric, $h_\gamma$ (restricted to the spheres of radius $r$) is simply a rescaling map from a round spheres to a round sphere, and thus preserve orthogonality.  Therefore,
$$\lb\ou(t_0),\ov(t_0)\rb=\lb d(h_\gamma)_{\og(0)}(\ou),d(h_\gamma)_{\og(0)}(\ov)\rb = \lb \ou,\ov\rb = 0.$$
It follows that
$$\lb \ov'(t),\ou(t)\rb = \lb T_{\ov(t)}\ox(t),\ou(t)\rb = (1/2)\frac{d}{dt}\lb \ov(t),\ou(t)\rb = 0.$$
This means that $\ov'(t)$ is parallel to $\ov(t)$, so we can write $\ov(t) = s(t)P(t)$ for some positive-valued function $s(t)$ and some parallel vector field $P(t)$.  Equation~\ref{kv} implies that $s''(t)\leq 0$ for all $t$.  Since $s(t)$ is nowhere zero, we conclude that $s(t)$ is constant.  Thus, $\ov(t)$ is a parallel Jacobi field.

It remains to explain why the Sharafutdinov fibers along $\gamma$ are mutually isometric.  Since $|\ov(t)|$ is constant, the $\{t=\text{constant}\}$ circles in the cylinder $C_r$ must all have the same length.  But these are great circle in the round spheres $\exp^\perp(\nu_{\gamma(t)}(\Soul))\cap B_r(\Soul)$, so these round spheres must all have the same diameters.  Since this is true for each $r$, the Sharafutdinov fibers along $\gamma$ have isometric distance spheres at all distances, so the fibers are mutually isometric.
\end{proof}


\bibliographystyle{amsplain}

\end{document}